\documentclass[review]{elsarticle}
\makeatletter
\def\ps@pprintTitle{%
 \let\@oddhead\@empty
 \let\@evenhead\@empty
 \def\@oddfoot{}%
 \let\@evenfoot\@oddfoot}
\makeatother

\usepackage{lineno,hyperref}
\modulolinenumbers[5]

\journal{Advances in Applied Mathematics}
 \usepackage{etex}
\usepackage[utf8]{inputenc}
\usepackage[T1]{fontenc}
\usepackage{geometry}
\usepackage{lmodern}
\usepackage{amsmath}
\usepackage{amsthm}
\usepackage{amssymb}
\usepackage{mathrsfs}
\usepackage{graphicx}
\usepackage{pst-all}
\usepackage{marvosym}
 \usepackage{url}

\newtheorem{theorem}{Theorem}[section]
\newtheorem{lemma}[theorem]{Lemma} \newtheorem{observation}[theorem]{Observation}
\newtheorem{corollary}[theorem]{Corollary}
\newtheorem{proposition}[theorem]{Proposition}

\newcommand{\N}{{\mathcal N}}

\newcommand{\M}{{\mathcal M}}
\newcommand{\cH}{{\mathcal H}}
\newcommand{\G}{{\mathcal G}}

\newcommand{\U}{{\mathcal U}}

\newcommand{\V}{{V^{\circ}}}

\title{Phylogenetic networks that are their own fold-ups
}
\author[uea]{K.T.~Huber\corref{cor}}
\ead{k.huber@uea.ac.uk}
\author[lirmm]{G.E.~Scholz}
\ead{g.scholz@uea.ac.uk}
\cortext[cor]{Corresponding author}
 \address[uea]{School of Computing Sciences, University of East Anglia,
      Norwich, UK.}
 \address[lirmm]{LIRMM, Universit\'e de Montpellier, CNRS, Montpellier, France.}
\date{\today}

\begin{document}

\begin{abstract}
Phylogenetic networks are becoming of increasing interest to evolutionary biologists due to their ability to capture complex non-treelike evolutionary processes. From a combinatorial point of view, such networks are certain types of rooted directed acyclic graphs whose leaves are labelled by, for example, species. A number of mathematically interesting classes of phylogenetic networks are known. These include the biologically relevant class of stable phylogenetic networks whose members are defined via certain ``fold-up'' and ``un-fold'' operations that link them with concepts arising within the theory of, for example, graph fibrations. Despite this exciting link, the structural complexity of stable phylogenetic networks is still relatively poorly understood. Employing the popular  tree-based, reticulation-visible, and tree-child properties which allow one to gauge this complexity in one way or another, we provide novel characterizations for when a stable phylogenetic network satisfies either one of these three properties.
\end{abstract}


\maketitle

\section{Introduction}

Phylogenetic networks are becoming of increasing interest to evolutionary
biologists due to their ability to capture
complex non-treelike evolutionary processes. Reflecting to some extent the different evolutionary
contexts within which such processes can arise has led to the
introduction of a number of mathematically interesting classes of
such structures \cite{G14,HRS10,St16}. These include the class of
stable phylogenetic networks which have already proven useful
for better understanding how, for example, polyploidy arose 
(see e.\,g.\,\cite{BOHMB07,HOLM06,LSHPOM09}).  

From  a combinatorial point of view, a phylogenetic network
is essentially a rooted directed acyclic graph
whose set of {\em leaves}, that is, vertices of indegree one and outdegree zero,
is labelled by a pre-given set $X$ of, for example, species. 
Now a phylogenetic network $N$ is called stable if
it can be thought of as the ``fold-up'' of a certain
multi-labelled tree $\U(N)$ into which a phylogenetic network
$N$ can be ``unfolded''\footnote{Note that the notions
  of a ``stable phylogenetic network'' and a
  ``nearly stable phylogenetic network'' 
  \cite{GGLVZ18} are unrelated
  -- see Section~\ref{sec:tree-child}
  for more on this.}.
Sometimes just called {\em MUL-trees}
\cite{CEJM13,HLMS08,HM06,HMSSS12},
such trees may be viewed as the phylogenetic
analogue of the universal cover of a
digraph \cite{HMSW16} and differ from the type of phylogenetic
trees commonly used by
evolutionary biologist by allowing the leaf set to be  a
multi-set rather than just a set. For example, for the phylogenetic
network $N$ depicted in terms of solid lines in Figure~\ref{favexpl}(i),
\begin{figure}[h]
	\begin{center}
		\includegraphics[scale=0.7]{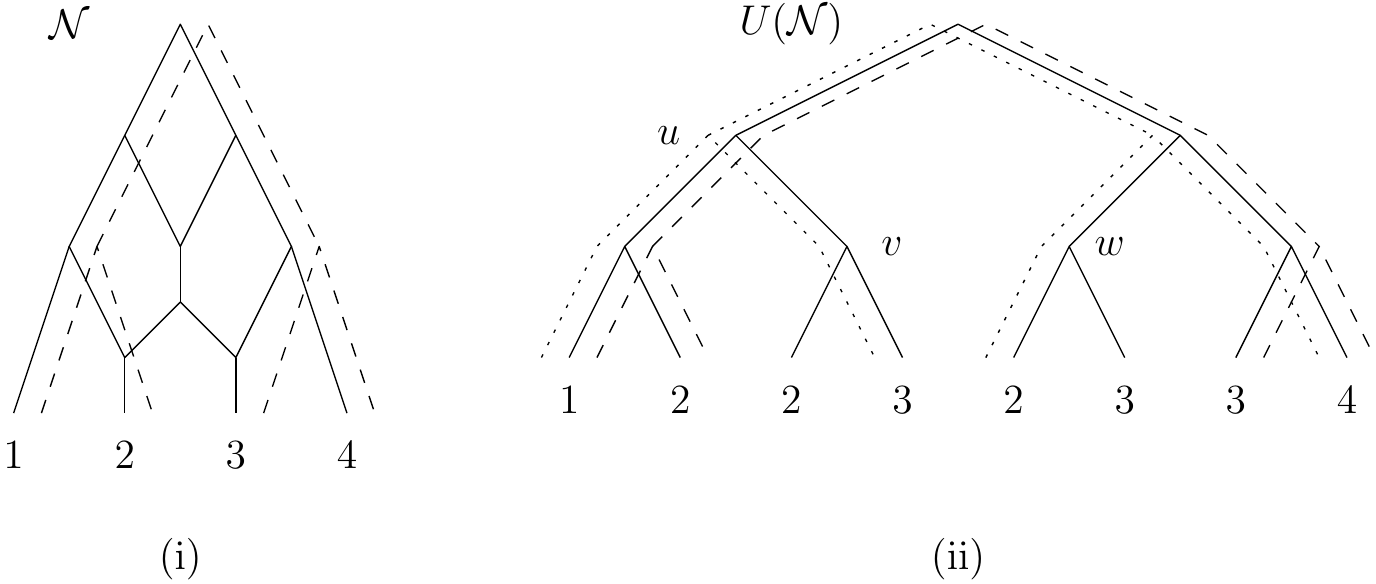}
		\caption{ 
			(i) A phylogenetic network $N$ on $X=\{1,2,3,4\}$ that is stable (solid lines). (ii) Ignoring the
			interior labels for the moment, the MUL-tree 
			$\U(N)$ obtained by unfolding $N$ (solid lines).
			Deferring the precise definitions to 
			the next sections, we indicate in (i)  
			a phylogenetic tree $T$ on $X$ in dashed lines 
			that is displayed by $N$. In addition, we indicate 
			in (ii) one of the ways $T$ is endorsed by $\U(N)$.
			 Note that the phylogenetic tree on $X$ given in dotted lines in (ii) is
			endorsed by $\U(N)$ but not displayed by $N$.
		}
		\label{favexpl}
	\end{center}
\end{figure}
 we picture 
 the tree $\U(N)$ in terms of solid lines in Figure~\ref{favexpl}(ii)
 -- see the next
section for a brief review of both operations and e.\,g.\,\cite{HMSW16} 
for recent results linking the un-fold and fold-up operations to concepts
arising in the theory of graph fibrations and also to the gene tree/species
tree reconciliation problem \cite{BS15, W13}. Despite these encouraging results,
it is however still largely unclear how structurally complex a stable
phylogenetic network can be. 

To help illustrate this question which lies in
the center of this paper, consider again
the phylogenetic network $N$ depicted in Figure~\ref{favexpl}(i). Then
$N$ is relatively
simple in the sense that it is ``tree-based''
\cite{FSS18, FS15, H16, Se16, Z16},
that is, it
can be thought of as a
rooted tree $T$ with leaf set $\{1, 2, 3, 4\}$ (called the 
base tree) to which arcs joining 
distinct arcs in $T$ have been added. As was shown in, for example,
\cite{JI17}
not every stable phylogenetic network need however be tree-based. On 
the other hand, if the complexity of $N$ is based on the notion of
``reticulation-visible'' \cite{BS16,GDZ17, vISS10} (essentially,
this means that for every vertex $h$ in a phylogenetic network $N'$
with indegree two or more there exists a leaf $x_h$ such
that $h$ lies on every directed
path from the root of $N'$ to $x_h$) then $N$
is not simple as the highest up vertex of $N$ with indegree two
does not satisfy that property. Denoting that vertex by $h$
it is clear however that $N$ can easily
be turned into a reticulation-visible network
on $\{1,\ldots, 5\}$ by subdividing the outgoing arc of $h$
by a new vertex $x$ and adding the arc $(x,5)$ and a leaf labelled
``5''. Since the resulting network is still stable,
it follows that there exist stable phylogenetic networks (that
are not also phylogenetic trees!) that are reticulation-visible. To help
establish our results which include characterizations
of stable phylogenetic networks that are tree-based or reticulation-visible,
we employ numerous maps between the vertex sets of
various graphs of interest to us, the most important of
which is a map $\overline{\xi_C^+}$ whose definition
we defer to Section~\ref{sec:injective} as it is quite involved.
For the convenience of the reader,
we summarize them in
terms of a commutative diagram in Figure~\ref{cdl}.

The outline of the paper is as follows. In the next section
we collect relevant basic definitions. In addition, we
review the aforementioned un-fold and fold-up
operations for phylogenetic
networks which underpin the definition of a stable
phylogenetic network. Focusing on such networks, we 
recall in Section~\ref{sec:injective}
the definition of a phylogenetic tree being
displayed by a phylogenetic network. Subsequent to this, we
characterize when such a network displays a phylogenetic tree in terms of when
our map $\overline{\xi_C^+}$ is injective (Theorem~\ref{thfm}). 
Using this insight, we then turn our attention to 
understanding three distinct popular properties of
phylogenetic networks in terms of properties of 
$\overline{\xi_C^+}$. These are the aforementioned tree-based
and reticulation-visible properties and the popular   
``tree-child'' property \cite{BST18,CRV07}. The latter essentially
means that for every non-leaf vertex of a phylogenetic network
at least one of its children has indegree one.
In particular, in Section~\ref{sec:bijective} we 
characterize when a phylogenetic tree
is a base tree of a stable phylogenetic network in terms of that
map (Theorem~\ref{thtb}) and in Section~\ref{sec:tree-child}, 
we characterize when such networks are tree-child or reticulation-visible
in terms of the map $\overline{\xi_C^+}$
(Theorem~\ref{thtc}). -- We refer the reader to \cite[Figure 10.12]{St16}
for a visualization of the interrelationships between network properties
such as the ones of interest to us for a certain type of
phylogenetic network. In
Section~\ref{sec:triplets-trinets}, we turn our attention
to the problem of characterizing stable phylogenetic
networks in terms of combinatorial
structures other than MUL-trees. We conclude with
Section~\ref{sec:discussion}
where we outline future directions of research
which might be of interest to pursue.

\section{Preliminaries}
In this section, we present relevant basic definitions. For this, 
we assume from now on that $X$ is a finite set of size at least three,
unless stated otherwise. We start with some basic concepts from graph theory. 

\subsection{Graphs and DAGs}
Suppose  $G$ is a rooted directed connected graph with vertex set $V(G)$
and arc set $A(G)$ that may or may not contain parallel arcs. 
Then we denote by $\rho_G$ the unique
root of $G$. If $a\in A(G)$ is an arc of $G$ then we 
denote by $tail(a)$ the start vertex of $a$ and by $head(a)$ the 
end vertex of $a$. Furthermore, we write $(x,y)$ for an 
arc $a$ with $x=tail(a)$ and $y=head(a)$. If $G$ does not
contain any directed cycles then we call it a 
{\em rooted directed acyclic pseudo-graph}, or {\em rooted 
	connected pseudoDAG}, for short. A rooted connected pseudoDAG
that does not contain any parallel arcs is called a {\em rooted directed 
	acyclic graph}, or {\em rooted connected DAG}, for short.
We refer to the replacement of an arc $a$ of $G$ by the directed path
$tail(a),w,head(a)$ where $w$ is a new vertex not contained in $V(G)$ 
as a {\em subdivision} of $a$. In this case we also call $w$ a 
{\em subdivision vertex} of $a$. A (pseudo)graph obtained from $G$ via a succession of arc-subdivision operations is called a {\em subdivision of $G$}.

Suppose for the following that
$G$ is a rooted connected pseudoDAG and that $v$ is a vertex of $G$. 
Then we call the number of incoming arcs of
$v$ the {\em indegree} of $v$, denoted by $indeg(v)$,
and its number of outgoing  arcs the {\em outdegree} of $v$, denoted
by $outdeg(v)$. We call a vertex with outdegree zero a
{\em leaf} of $G$ and  a vertex of $G$ that is not a leaf an 
{\em interior vertex} of $G$. The set of all leaves of $G$ is
denoted by $L(G)$ and the set of all interior vertices of $G$
is denoted by $V^{\circ}(G)$. Moreover, we call a vertex $v$ of $G$ a
{\em tree vertex}
if $indeg(v) \leq 1$, and a {\em hybrid vertex} of $G$ if $indeg(v) \geq 2$.
We denote the set of tree vertices of $G$ by $V^-(G)$.

We say that a vertex  $w$ of $G$ is {\em below} a vertex $v$ of $G$ 
 if there exists a directed path
 from $v$ to $w$ (note that a vertex $v$ is below itself).
 If, in addition, $v\not=w$ then we say that $w$ is {\em strictly below}
 $v$. If $w$ is below  $v$ then we call $v$ an
{\em ancestor} of $w$. In case $G$ is a tree, we call for
any set $Y\subseteq L(G)$ of size two or more, the unique last vertex of $G$ that simultaneously lies on every
directed path from $\rho_G$ to an element in $Y$ the
{\em last common ancestor} of the elements of $Y$. We denote this vertex by $lca_G(Y)$ (where we omit the index $G$ if
the graph considered is clear from the context).

\subsection{Phylogenetic networks and MUL-trees}
Inspired by \cite{HMSW16}, we define
a {\em $X$-network $N=(N,\nu)$} to be an
ordered pair consisting of a 
rooted connected pseudoDAG $N$ such that
$\nu:X\to L(N)$ is a bijective map from $X$ into the leaf set of $N$ and
(i) the root $\rho_N$ of $N$ has indegree
zero, (ii)  every non-leaf tree vertex has
outdegree two, (iii) every hybrid vertex has outdegree one, and
(iv) all leaves have degree one. To help keep notation at
bay, we always assume that the leaf set of an
$X$-network $\N=(N,\nu)$ is in fact $X$ (implying that 
$\nu$ is the identity map on $X$). In case there is no
confusion, we therefore denote an $X$-network $\N=(N,\nu)$
simply by $N$.
If $N$ is in fact a DAG and therefore does not contain parallel arcs, then
we call $N$ a {\em phylogenetic network (on $X$)}\footnote{Phylogenetic
	networks that enjoy Property (ii) were called {\em semi-resolved}
	phylogenetic networks in \cite{HMSW16}.}.  Note that in case
        Property (iii) is strengthened by requiring that
        every hybrid vertex has indegree two then
        $N$ is called a {\em binary} phylogenetic network.
        Note that a phylogenetic network on $X$ that does not contain
        a hybrid vertex is generally called a {\em phylogenetic tree (on $X$)}
        (see e.\,g.\,\cite{SS03}). Also note that since every non-leaf
tree vertex of a phylogenetic network (on $X$) has outdegree two
it follows that a phylogenetic tree (on $X$) must be binary.
        A phylogenetic tree $T$ on $X=\{a,b,c\}$
  is called a {\em triplet (on $X$)}. If $lca_{T}(\{a,b\})$ is strictly
  below $lca_{T}(\{a,c\})$ then we denote $T$ by $ab|c$.
 
Motivated by the definition of a phylogenetic tree on $X$, 
we define for $|X|\geq 1$ the
{\em MUL-tree $\M$ on ($X$)} to be an ordered
pair $(M,\mu)$ where $M$ is a rooted tree
for which $indeg(\rho_M)=0$ holds and no vertex
has indegree and outdegree one, and $\mu:X\to \mathcal P(L(M))$
is a {\em labelling map}
from $X$ into the set $\mathcal P(L(M))$ of all non-empty
subsets of the leaf set $L(M)$ of $M$
such that the following properties are satisfied:
(i) For every leaf $l\in L(M)$ there exists some $x\in X$ with 
$l\in \mu(x)$, and (ii) there exists no leaf $l\in L(M)$ and $x,x'\in X$
distinct such that $l\in \mu(x)\cap \mu(x')$ -- see 
Figure~\ref{favexpl}(ii) for an example of such a tree for $X=\{1,2,3,4\}$.
Informally speaking, Properties (i)
and (ii) state that every leaf of $M$ must be labelled by precisely
one element in $X$. Note that if, for all $x\in X$, the size
of $\mu(x)$ is one and $|X|\geq 3$ then we implicitly identify $\mu(x)$
with its unique element, thus rendering $M$ a phylogenetic tree on $X$
in case $M$ is binary.

Suppose for the following that $\M=(M,\mu)$ is a MUL-tree on $X$.
If $\M'=(M',\mu')$ is a further MUL-tree on $X$
then we say that $\M$ and $\M'$
are {\em isomorphic} if there exists a bijection $\eta: V(M)\to V(M')$
that induces a graph isomorphism between $M$ and $M'$ such that
$\mu'(x)=\{\eta(y)\,:\, y\in \mu(x)\}$ holds
for all $x\in X$.
For a vertex  $v\in V(M)$ that is not the
root of $\M$ we denote by
$\M_v=(M_v,\mu_v)$  the MUL-tree with root $v$ obtained
from $\M$ by deleting the incoming
arc of $v$ in $M$ and, to obtain $\mu_v$, restricting $\mu$
to the subset of $X$ that contains the leaf labels of the
set of leaves of $M$ below $v$.  We call a MUL-tree
$\M'=(M',\mu) $ a {\em subMUL-tree} of $\M$
if $\M'$ is either $\M$ itself or 
there exists some $v\in V(M)$ such that $\M'$ and $\M_v$ are
isomorphic.

To help keep notation at bay, we also denote from now on
the vertex set and arc set of 
a MUL-tree $\M$ by $V(\M)$ and $A(\M)$,
respectively.

\section{Stable phylogenetic networks}\label{sec:stable}

In this section, we present a formal definition of a
stable phylogenetic network. Since this definition
relies on a certain ``un-fold'' operation
for phylogenetic networks to obtain a MUL-tree
and a certain ``fold-up'' operation for MUL-trees to obtain an $X$-network,
we briefly review these constructions first. For details on both
of them, we also refer the interested reader 
 to \cite{HM06} (see also \cite{HMSW16}).

\subsection{The un-fold of a phylogenetic network}
To be able to outline this construction, we require further notation.
Suppose $N$ is a phylogenetic network on $X$. 
The set of all directed paths of $N$ starting 
at $\rho_{N}$  and ending in a vertex of $V(N)$ is denoted
by $\pi(N)$. 
For an element $x \in X$, we denote by $\pi_x(N)$ the subset of directed
paths of $\pi(N)$ that end in $x$. Finally, we put 
$\pi_0(N)=\bigcup_{x \in X} \pi_x(N)$.

Viewable as a three-step process, the construction
of the un-fold ${\mathcal U}(N)=(U(N),\phi_{N})$ of $N$ works as follows.
In the first step, we construct
a tree $U^*(N)$
from $N$ whose vertices are the elements in 
$\pi(N)$ and whose arcs are the pairs
$(P, P')$ in  $\pi(N)\times \pi(N)$ for which there exists an
arc $a\in A(N)$
such that $tail(a)$ is the end vertex of $P$, the end vertex
of $P'$ equals $head(a)$,
and $P'$ is the path $P$ extended by the arc $a$.
To obtain the tree $U(N)$, the vertices of  $U^*(N)$ with indegree 
and outdegree one are then suppressed, where by
{\em suppressing} such a vertex $v$, we mean deleting $v$ as well
as the incoming and
outgoing arcs of $v$, and adding an arc from the parent of $v$ to the
child of $v$. Observe
that the suppressed vertices of $U^*(N)$
are precisely the paths in $\pi(N)$ ending at a hybrid vertex
of $N$ \cite{HMSW16}.

The next and final  step of the construction of $\U(N)$
is concerned with defining the map
$\phi_N:X\to \mathcal P(L(U'))$ where we put $U'=U(N)$.
Since, as was observed above,  hybrid vertices are the
only vertices of $N$ with a single outgoing arc,  
there exists a bijection $\Psi_N: \pi^-(N)\to V(U')$ 
from the set $\pi^-(N)$ of directed paths in $\pi(N)$ that end in a 
tree vertex to the vertex set of $U'$. Thus, the 
restriction $\psi_N=\Psi_N|_{\pi_0(N)}$ 
of $\Psi_{N}$ to $\pi_0(N)$ induces a bijection 
from  $\pi_0(N)$ to $L(U')$. 
Putting $\phi_{N}(x)=\{l\in L(U'): \psi_{N}^{-1}(l)\in \pi_x(N)\}$, 
for all $x\in X$, completes the construction
of $\U(N)$.\\

\subsection{The fold-up of a MUL-tree}
\label{sec:subsec-fold-up}
To state this construction, we again require further
concepts which we introduce next. Suppose $\G=(G,\gamma)$ is a pair
consisting of a rooted connected pseudoDAG $G$ and a labelling map
$\gamma:X\to \mathcal P(L(G))$. If $G$ contains a cut-arc $a$
then we denote by $G(a)$
the connected component that contains $head(a)$ in its vertex set
 when deleting $a$. If $G(a)$ is a rooted tree
then we denote the MUL-tree induced by $G(a)$ by $\G(a)$.
By abuse of terminology, we say that $\G(a)$ is a subMUL-tree of $\G$
even if $\G$ itself is not a MUL-tree.

Suppose for the following that $a$ is a cut-arc of $\G$
such that $\G(a)$ is a subMUL-tree of $\G$.
We say that $\G(a)$ is {\em inextendible}
if there exists a cut arc $a'$ in $\G$ distinct from $a$ such that
the subMUL-trees $\G(a)$ and $\G(a')$ are isomorphic.
We denote by $S_{\G(a)}$ the set of all
subMUL-trees of $\G$ that are isomorphic with $\G(a)$ (including
$\G(a)$ itself).
In case $\G(a)$ is inextendible
then we say that $\G(a)$ is {\em maximal inextendible}
if every inextendible subMUL-tree
$\cH$ of $\G$ that contains an element in 
$S_{\G(a)}$ as a subMUL-tree is contained in $S_{\G(a)}$.
 In other words, $\G(a)$ is maximal inextendible
if there is no other inextendible subtree that properly contains a
subMUL-tree that is isomorphic to $\G(a)$.

For  $\M=(M,\mu)$ a MUL-tree on $X$,
we define the \emph{fold-up} $F(\M)$ of $\M$
to be the $X$-network obtained by applying
the following operation until no inextendible subMUL-tree
remains in $\M$. First select a maximal inextendible
subMUL-tree $\M'$ from $\M$. Then subdivide the
incoming arc of the root of each subMUL-tree  in $S_{\M'}$
with a new vertex. Next, identify all subdivision vertices
introduced at this step
and delete all but one subMUL-tree in $S_{\M'}$ from $\M$
and also the incoming arcs of their roots. Finally, replace $\M$ by the
resulting rooted labelled pseudoDAG and repeat the above
process for a maximal inextendible subMUL-tree $\M'$ of $\M$.

Note that upon completion, the resulting pseudoDAG $F(\M)$ is
always an $X$-network and that this network is independent of the order
in which the maximal inextendible subMUL-trees have been processed
\cite{HM06}. However, $F(\M)$ 
 need not be a phylogenetic network on $X$ as it might have
parallel arcs. We call a MUL-tree $\M$ {\em sound}
if $F(\M)$ is a phylogenetic network. Note that
stable binary phylogenetic networks were characterized
in  \cite[Theorem 1]{HMSW16} in terms of a vertex property.
Also note that sound MUL-trees were characterized in \cite{HMSW16}
as those MUL-trees which do not contain a pair of isomorphic
subMUL-trees whose roots share a parent. So, for example, the MUL-tree
depicted in Figure~\ref{figfun}(ii) is sound.

To help illustrate the fold-up operation, consider
the MUL-tree $\M$ pictured in Figure~\ref{figfun}(ii). 
Then the rooted, connected pseudoDAG  on $X=\{1,2,3\}$ depicted 
in Figure~\ref{figfun}(iii) is $F(\M)$. Clearly,
$F(\M)$ is a phylogenetic network on $X$ and 
the MUL-trees $\U(F(\M))$ and $\M$ are isomorphic.
Note however that $\M$ is
also the un-fold of the $X$-network $N$ pictured in Figure~\ref{figfun}(i)
and $F(\M)$ and $N$ are clearly not isomorphic. We call
an $X$-network $N$  that is isomorphic with 
$F(\U(N))$ {\em stable}. Note that for any stable phylogenetic
network $N$ the un-fold $\U(N)$ of $N$ is a sound MUL-tree.

\begin{figure}[h]
\begin{center}
\includegraphics[scale=0.7]{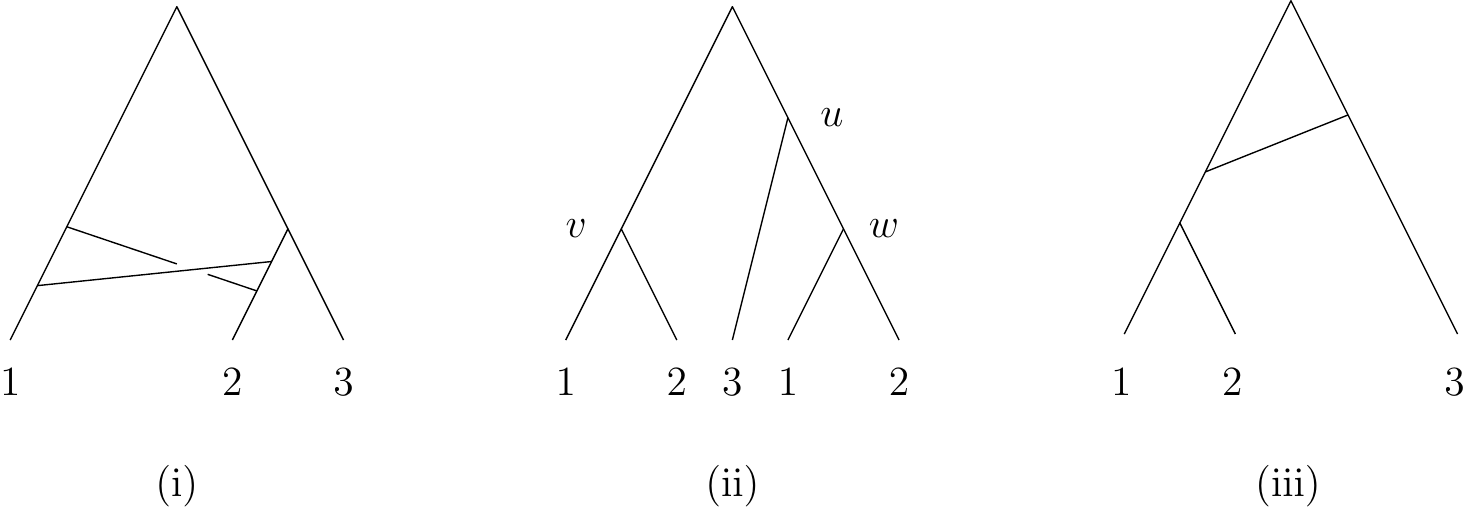}
\caption{(i) A phylogenetic network $N$ on $X=\{1,2,3\}$ that is not stable. (ii) The 
  unfolded version $\U(N)$ of $N$ (ignoring interior labels for the
  moment). (iii) The folded version $F(\U(N))$ 
  of $\U(N)$.}
\label{figfun}
\end{center}
\end{figure}

As we shall see, it is advantageous to describe the
set of tree vertices of a folded-up MUL-tree $\M$
in terms of an equivalence relation  $\sim$ on
the vertex set of $\M$. To state that relation, suppose that 
 $\M=(M,\mu)$ is a MUL-tree. 
Let $\sim$ denote the  equivalence relation
defined, for all $u,v\in V(M)$, by putting $u\sim v$
if $u=v$ or the  subMUL-trees
$\M_v$ and $\M_u$ of  
$\M$ rooted at $u$ and $v$, respectively, are isomorphic.
Note that if $v$ and $w$ are two leaves of $\M$ then
$v\sim w$ if and only if there exists some $x\in X$ such that
$v,w\in \mu(x)$. We denote the quotient of $V(\M)$ via 
$\sim$ by $V(\M)/\sim$, and we denote by $\overline{a}$
the equivalence class
generated by a vertex $a\in V(\M)$. We call the map 
$p_{\M}:V(\M)\to V(\M)/\sim$ defined by mapping
each vertex of $M$ to its equivalence class under $\sim$ the
{\em projection of $\M$ to  $V(\M)/\sim$}. Note that if $u$ and $v$ 
are vertices in $\M$ such that $u$ is an ancestor of
$v$, then $|p_{\M}(u)| \leq |p_{\M}(v)|$.

For example, for $v$ and $w$ the indicated interior vertices in the
MUL-tree depicted in Figure~\ref{favexpl}(ii), we have  $w\sim v$.
Furthermore, $ |V(\M)/\sim|=10 $ and four of those 10 equivalence
classes are induced by the leaves of $\M$. Finally, 
$|p_{\M}(v)|=2$  and $|p_{\M}(u)|=1$.

As a first consequence, we have the
following result whose straight-forward proof we
leave to the interested reader.

\begin{lemma}\label{lmsim}
  Suppose that $\M=(M,\mu)$ is a sound MUL-tree and that $u,v\in V(M)$.
  Then, $u \sim v$ if and only if
  the directed paths $\Psi_{F(\M)}^{-1}(u)$ and 
  $\Psi_{F(\M)}^{-1}(v)$ in $F(\M)$
  have the same end vertex (which
  must necessarily be a tree vertex).
In particular, if there exists a  stable phylogenetic network
$N$ such that 
$\M$  and $\U(N)$ are isomorphic then there 
exists a trivial bijection $\kappa_{N}:V^-(N)\to V(\U(N))/\sim$
from the set $V^-(N)$ of tree vertices of $N$ to the projection of
$V(\U(N))$ to $V(\U(N))/\sim$.
\end{lemma}

Note that Figure~\ref{figfun} indicates that
the second part of Lemma~\ref{lmsim} does not hold if
we drop the assumption that $N$ is a  {\em stable} phylogenetic network.


To be able to state an observation concerning
stable phylogenetic networks which might be of interest in its own right,
we require a further concept.
Suppose $G=(V,E)$ and $G'=(V',E')$ are two
directed graphs. Then we define the {\em union} of $G$ and $G'$ 
to be the directed graph with vertex set $V\cup V'$ and arc set $E\cup E'$.
We call the union of two arc-disjoint directed paths $P$ and $P'$ of $G$
which have the same start vertex and also the same end vertex a
{\em reticulation cycle} of $G$. 

\begin{observation}\label{isosub}
  Suppose that $N$ is a stable phylogenetic network and
  that $P_1, P_2\in \pi^-(N)$ distinct. Then the union of
$P_1$ and $P_2$ contains a subgraph that is a
reticulation cycle of $N$ if and only if the vertices
$\Psi_{N}(P_1)$ and $\Psi_{N}(P_2)$ of $\U(N)$ are
contained in two distinct but isomorphic subMUL-trees of $\U(N)$. 
\end{observation}

We again leave the straight-forward proof of the
observation to the interested reader.


\section{Displaying and endorsing phylogenetic trees}\label{sec:injective}
 

As suggested by Figure~\ref{favexpl}, any phylogenetic tree  that is
``contained'' in a phylogenetic network $N$ is also ``contained''
(in possibly more than one way) in its associated MUL-tree $\U(N)$.
However the converse need not hold.
Understanding this relationship is the main purpose of this section.

We start with formalizing
the idea of ``containment'' in an $X$-network and in a MUL-tree,
respectively.
Suppose for the following that $N$ is a $X$-network. 
If $N'$ is a further $X$-network then we say that $N$ 
and $N'$ are {\em isomorphic} if there exists a bijective map
$f:V(N)\to V(N')$ that induces 
a graph isomorphism between $N$ and $N'$ that maps every element of $X$
to itself. Following \cite{HMSW16}, we say that $N$
{\em displays} a phylogenetic tree $T$ on $X$ if there exists
a subgraph $N'$ of $N$ with leaf set $X$ that is isomorphic with
a subdivision $T'$ of $T$ and that isomorphism is the identity on $X$.
Informally speaking, this means that $N'$ can be obtained from
$T$ by replacing arcs $a$ of $T$ by directed paths from
$tail(a)$ to $head(a)$. It should however be noted 
that this definition does not imply that
the root of a phylogenetic tree $T$ displayed by a
 phylogenetic network $N$ is also the root of $N$. 
 Furthermore, it is worth noting that any phylogentic network gives rise to a
 phylogenetic tree that is displayed by it.
 
 To illustrate the concept of displaying, consider 
the triplet $\tau=3|12$ and 
the phylogenetic network $N$ depicted in Figure~\ref{figfun}(iii).
Then $\tau$ is displayed by $N$  in two different ways one of which uses the
 root $\rho_N$ of $N$ as the root of $\tau$ and the other uses a
 child of $\rho_N$ as the root of $\tau$.

Suppose that  $\M=(M,\mu)$ is a MUL-tree on $X$ where $|X|\geq 1$.
Then we call a set $C \subseteq L(M)$ an \emph{$X$-set 
  (of $\M$)} if $|C\cap \mu(x)|=1$ holds for all $x \in X$. 
Informally speaking, an $X$-set for $\mathcal M$ contains for all
$x\in X$ exactly one leaf of $M$ 
labelled by $x$. Note that if $|X|\geq 3$ then any
$X$-set $C$ of $\M$ induces a phylogenetic tree 
$M_C$ on $X$ by first constructing a tree $M_C^+$ from $M$ that is
spanned by all leaves in $C$ and then, to obtain $ M_C$,
suppressing all resulting vertices of indegree and outdegree one
-- see again Figure~\ref{tmc} for an illustration
of this construction. More formally, 
$V(M_C^+)$ is the set of all $v\in V(\M)$ that
lie on a directed path from $lca_{M}(C)$ to a leaf  in $C$ and an
arc $a\in A(\M)$ is an arc of $M_C^+$ if there exists a leaf $l\in C$
such that $a$
is crossed by a directed
path from $lca_M(C)$ to $l$.

Note that for any
phylogenetic tree $T$ displayed by a phylogenetic network $N$,
there exists an $X$-set $C$ of $\M=\U(N)$ such that $T$ and
$M_C$ are isomorphic. However,
unless the vertices of $T'$ corresponding to hybrid
vertices of $N$ are suppressed, the subdivision 
$T'$ of $T$ referred to in the definition of displaying is in general
not isomorphic with  $M_C^+$.
Note that in case they are suppressed, this isomorphism can be defined
in such a way that it is the
identity on the leaf sets -- see, for example, Figure~\ref{tmc}.

\begin{figure}[h]
\begin{center}
\includegraphics[scale=0.7]{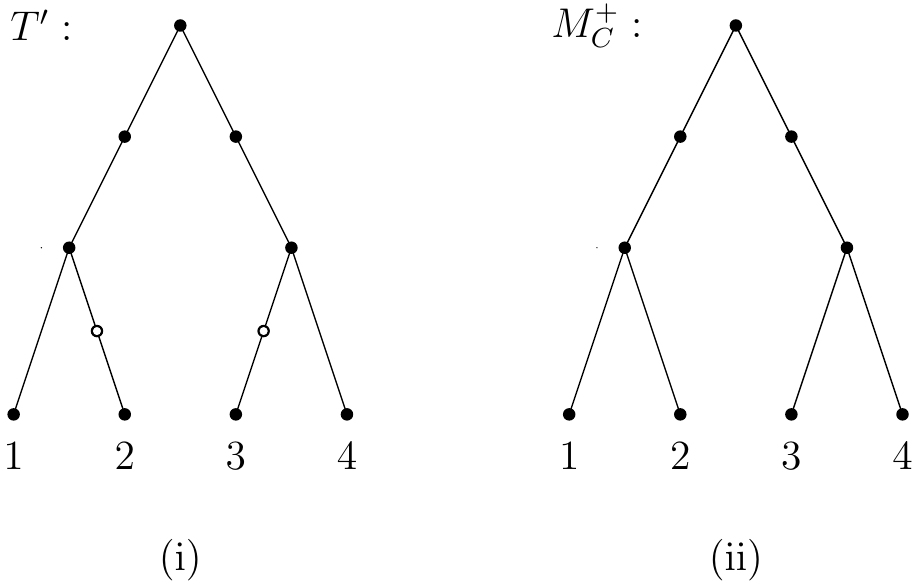}
\caption{(i) For the tree $T$ displayed by $N$ in terms of the dashed
  arcs in Figure~\ref{favexpl}(i), we depict the subdivision $T'$
  of $T$  that is isomorphic to a subgraph of $N$.
  The vertices corresponding to non-leaf tree-vertices (resp. hybrid
  vertices) of $N$ are indicated by
  black (resp. empty) dots. (ii) For $C$ the leaves of $\M$ that are
  incident with a dashed arc in  Figure~\ref{favexpl}(ii), 
  the tree $M_C^+$ endorsed by $\M$. Vertices in $M_C^+$ that
  correspond to  non-leaf vertices of $\M$ are marked with black dots --
  see text for details.
}
\label{tmc}
\end{center}
\end{figure}

We say that a phylogenetic tree $T$ is {\em endorsed
by a MUL-tree $\M=(M,\mu)$ via some $X$-set $C\subseteq L(M)$}
if $T$ and $M_C$ are isomorphic. In case, the $X$-set $C$ is clear
from the context, then we simply say that
$T$ is {\em endorsed by $\M$}. Note that
in \cite{HMSW16} a phylogenetic tree
that is endorsed by a MUL-tree $\M$ was
said to be {\em weakly displayed} by the fold-up of $\M$.

To illustrate this definition, consider
the stable phylogenetic network $N$ on $X=\{1,2,3,4\}$
depicted in bold lines in Figure\,~\ref{favexpl}(i).
Then the dotted phylogenetic tree $T$ in Figure~\ref{favexpl}(ii)
is endorsed by $\U(N)$. However $T$ is not displayed
by $N$ since
the joint grandparent of the leaves $2$ and $3$ in $N$
is used by the path from the root $\rho_{T}$ of $T$
to the leaf labelled 2 and
also by the path from $\rho_{T}$ to the leaf labelled 3.
Intriguingly, the phylogenetic tree in dashed lines in
(ii) is endorsed by $\U(N)$ and is not only displayed by $N$
but also is a base tree for $N$
(see Theorem~\ref{thtb} for more on this).

Note that the  process of constructing the phylogenetic tree
$M_C$ from a phylogenetic tree on $X$
endorsed by a MUL-tree $\M=(M,\mu)$
via some $X$-set $C$ trivially induces an
ancestor-relationship-preserving injective map
$$
\xi_C: V(M_C) \to V(\M)
$$
such that $\xi_C(x)$ is the unique element in
$C\cap \mu(x)$, for all $x \in X$.
Let
$
\xi^+_C: V(M_C^+) \to V(\M)
$
denote the canonical extension 
of $\xi_C$ to $V(M_C^+) $.

The following result is implied by \cite[Theorem 6]{HMSW16}
which characterizes phylogenetic trees that are weakly displayed by
a phylogenetic network. As such, it presents a first
link between displaying a phylogenetic tree by a network
and endorsing it by the un-fold of the network.
To state it, we require a further map which we define next.

Suppose that $N$ is a phylogenetic network on $X$
and $C$ is an $X$-set of $\M=\U(N)$. Then for any
tree $T'$ obtained from $M_C^+$ by suppressing some (but not necessarily all!)
of its vertices of indegree and outdegree one
it follows that there exists an injective  map 
$P_{T'}:V(T') \to \pi^-(N)$ which associates to each 
vertex $v$ of $T'$ a (necessarily unique) 
directed path $P_{T'}(v) \in \pi^-(N)$. Note that in
case $M_C$ is displayed by $N$ there exists a subgraph
$N'$ of $N$ that is a subdivision of $M_C$ such that,
for all vertices $v\in V(T')$,
the directed path $P_{T'}(v)$ does not cross any arc of $N$ that
is not also contained in $N'$.

\begin{lemma}\label{lmwd}
Suppose that $N$ is a phylogenetic network on $X$ and that $T$ is a 
phylogenetic tree on $X$ displayed by $N$. 
Then $T$ is endorsed by $\U(N)$ via some $X$-set $C$ of $\U(N)$. Moreover, 
$C$ can be chosen in such a way that 
$\Psi_{N}\circ P_{T}=\xi_C$.
\end{lemma}


Note that it
is not difficult to see that the converse of Lemma~\ref{lmwd}
need not hold.




To be able to establish Theorem~\ref{thfm} which is the main
result of this section, we next introduce two crucial maps.
Suppose that $N$ is a phylogenetic network on $X$ and $T$
is a phylogenetic tree  on $X$ that is displayed by
$N$. Let $C$ be an $X$-set of $\M=\U(N)$ such that $T$ is endorsed
by $\M$ via $C$. Then we put 
$$
\overline{\xi_C}:V(M_C)\to V(\M)/\sim ;\,\,\,\, v\mapsto p_{\M}\circ \xi_C(v)
$$
and 
$$
\overline{\xi_C^+}:V(M_C^+)\to V(\M)/\sim ;\,\,\,\, v\mapsto
p_{\M}\circ \xi_C^+(v).
$$
%
Note that the latter map was already mentioned in the introduction
and is central for the remainder of the paper.

In addition, we denote by $end:\pi^-(N)\to V^-(N)$ a map that
associates to a path $P\in \pi^-(N)$ the tree vertex of $N$
in which $P$ ends. Also, we denote by $f$
the map
$end \circ \Psi_{N}^{-1} :V(\M)\to V^-(N)$ and
by $\iota_1:V(M_C)\to V(M_C^+)$ the map that
maps every vertex in $V(M_C)$ to itself.

For $N$ a stable phylogenetic network on $X$ and $T$ a phylogenetic
tree on $X$ that is displayed by $N$, we next
summarize the main maps considered in this
paper in Figure~\ref{cdl} for the convenience of the reader.
Note that, by Lemma~\ref{lmwd}, $T$ is also endorsed by $\U(N)$
and that
the $X$-set $C\subseteq L(\U(N))$ is chosen
in such a way that $T$ and $M_C$ are isomorphic and that
$\Psi_{N}\circ P_{T}=\xi_C$.

\begin{figure}[h]
\begin{center}
\includegraphics[scale=0.7]{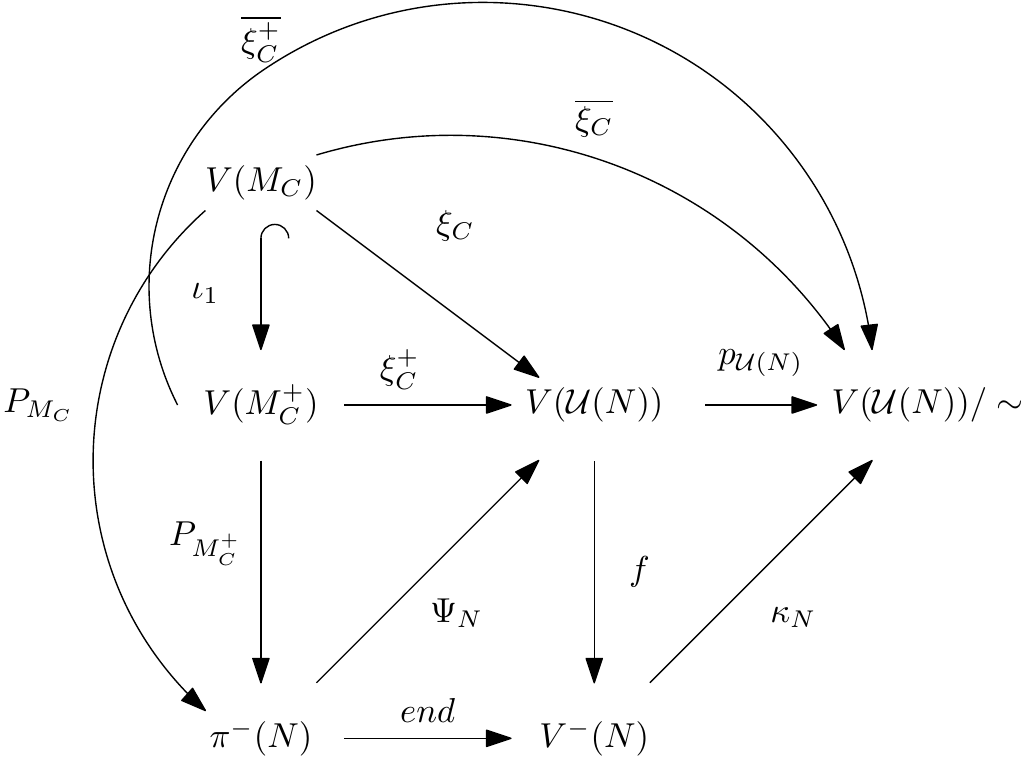}
\caption{For $N$ a stable phylogenetic network on $X$
  and $C$ an $X$-set of $\U(N)$ such that $M_C$ is displayed by $N$,
we depict the main maps that play a role in this paper in the form of a
commutative diagram. Note that since $N$ is stable, the phylogenetic
networks $F(\U(N))$ and $N$ are isomorphic.
}
\label{cdl}
\end{center}
\end{figure}

Armed with these definitions we are now ready 
to state the main result of this section. Informally speaking,
it ensures that a phylogenetic 
tree $T$ is displayed by a stable phylogenetic network $N$ if and only if 
$T$ can be embedded (in the graph-theoretical sense)
by $\overline{\xi_{C}^+}$ into  $\U(N)$ so that it is endorsed by
$\U(N)$ via some subset $C\subseteq L(\U(N))$ 
.

\begin{theorem}\label{thfm}
Suppose that  $N$ is a stable phylogenetic network on $X$ and that  $T$ is 
a phylogenetic tree on $X$. Then, $T$ is displayed by $N$ if and 
only if there exists an $X$-set $C$ of $\U(N)$ that endorses $T$ and 
the map $\overline{\xi_{C}^+}$ is injective.
\end{theorem}

\begin{proof}
  Note first that since $N$ is stable we may assume without
  loss of generality that 
$N$ is in fact the fold-up $F(\M)$ of a sound MUL-tree $\M=(M,\mu)$ on $X$.
  In view of Lemma~\ref{lmsim}, we therefore have that the map
  $\kappa_N:V^-(N)\to V(\M)/\sim$
in Figure~\ref{cdl} is bijective.
Also note that we may assume without loss of generality that $N$ is not
a phylogenetic tree on $X$ as otherwise the theorem clearly holds.

Assume first that $T$ is displayed by $F(\M)$.
Then, by Lemma~\ref{lmwd}, $T$ is endorsed by $\M$ via some
$X$-set $C$ of $\M$. Thus, $T$ and $M_C$ must be isomorphic.
Without loss of generality, we may therefore assume
that $T$ is in fact $M_C$. Furthermore
and again by Lemma~\ref{lmwd} we may assume that $C$ is such
that $\Psi_{N}\circ P_{M_C}=\xi_C$.

To see that the  map $\overline{\xi_{C}^+}$ is injective assume
 that there exist $w,w'\in V(M_C^+)$ such that
 $\overline{\xi_{C}^+}(w)=\overline{\xi_{C}^+}(w')$.
 Then $\xi_{C}^+(w) \sim \xi_{C}^+(w')$. By Lemma~\ref{lmsim} combined with the
 assumption that  $\Psi_{N}$ is bijective it follows that 
the directed paths $\Psi_{N}^{-1}(\xi_{C}^+(w))$ and
$\Psi_{N}^{-1}(\xi_{C}^+(w'))$
 must end in the same tree vertex $v$ contained in $V^-(F(\M))=V^-(N)$.
Since $w=end\circ P_{M_C^+}  (w)=end\circ \Psi_{N}^{-1}(\xi_{C}^+(w))=v=
end\circ \Psi_{N}^{-1}(\xi_{C}^+(w'))=end\circ P_{M_C^+}  (w')=w'$
 it follows that $\overline{\xi_{C}^+}$ is injective.

Conversely, assume that there exists an $X$-set $C$ of $\M$ that endorses 
$T$ and $\overline{\xi_{C}^+}$ is injective. Then we may assume again without
loss of generality that $T$ is in fact $M_C$. Since  $\overline{\xi_{C}^+}$
is injective and $\kappa_{N}:V^-(N)\to
V(\M)/\sim$ is bijective, it follows that the map 
$f\circ\xi_C^+:V(M_C^+)\to V^-(N)$ is injective. Hence, any two distinct
vertices in $V(M_C^+)$ get mapped to two distinct tree vertices of $N$.
Since $f\circ\xi_C^+$ preserves the ancestor relationships in $M_C^+$
it follows that $f\circ\xi_C^+$ induces an embedding of $M_C^+$ into $N$.
Since $M_C^+$ is a subdivision of $M_C$ it follows that
$M_C$ is displayed by $N$.
\end{proof}

As it turns out, the equivalence postulated in Theorem~\ref{thfm} 
does not hold in case $N$ is not stable. To illustrate this fact, consider 
the phylogenetic network $N$
depicted in Figure~\ref{figfun}(i). Then $N$ is not
stable as $F(\U(N))$ is the phylogenetic network
pictured in Figure~\ref{figfun}(iii). As is easy to
check the triplet $\tau=23|1$ is displayed by $N$.
Also there is a unique choice of the $X$-set $C$
of $\U(N)$ such that $\tau$ is endorsed by
$\U(N)$ via $C$. However for that choice of $C$ the map
$\overline{\xi_C^+}$ is not injective as
$\kappa_{N}^{-1}\circ\overline{\xi_C^+}$ maps the
vertices $v$ and $w$ of $\tau$ indicated
in $\U(N)$ to the unique parent of leaves 1 and 2 in
$F(\U(N))$.

\section{Stable phylogenetic networks that are tree-based}
\label{sec:bijective}


As was observed in e.\,g.\,\cite{St16}, some (but not all!) phylogenetic
networks can be thought of as phylogenetic trees
to which additional arcs have been added carefully. Motivated by
this observation, we next  
turn our attention to characterizing such networks in case
they are stable.
We start with introducing terminology from
\cite{JI17} which we present within our framework.

Suppose that $N$ is
a phylogenetic network $N$ on $X$. Then we say
that $N$ is {\em tree-based} if there exists a phylogenetic
tree $T$ on $X$, called
the {\em base tree of $N$}, such  that
$N$ can be obtained from $T$ by applying the following $3$-step process. 
First, subdivide some of the arcs of $T$ by
adding new vertices to them to obtain a
rooted tree whose leaf set is $X$.
Next, join these subdivision vertices by adding arcs between them,
ensuring that no parallel arcs or directed cycles are created.
Finally, suppress all subdivision vertices in the resulting rooted DAG
whose indegree and outdegree is one. Note that, in addition to
this definition of tree-based, \cite{JI17} also present an alternative
equally attractive definition of a non-binary tree-based phylogenetic network.

Clearly, a tree-based phylogenetic network may 
have more than one base tree. Also, all base trees 
of a tree-based network $N$ are displayed by $N$, 
although the converse is not true in general.
An example for this is furnished by the
dashed phylogenetic tree $T$ on $\{1,\ldots, 4\}$
in Figure~\ref{favexpl}(i),
which is clearly displayed by
the phylogenetic network $N$ depicted in that figure.
However, although  $T$ is a base tree for $N$ it is easy
to check that the embedding of $T$ given in terms of
the dashed edges does not lend itself as the starting point
for constructing $N$ from $T$ as described above 
 (see, for example, \cite{Se16}
for a discussion on the relationship between a tree-based
phylogenetic network and a phylogenetic tree displayed
by a phylogenetic network). As
is easy to check, if the vertex set of $T$ is mapped into
the vertex set of $N$ such that all tree vertices of $N$
are used then $T$ is not only displayed
by $N$ but also is a base tree of
$N$. To demonstrate that this relationship is not a coincidence
is the purpose of the next result. To be able to establish it,
we say that a phylogenetic network $N$ is {\em compressed} if 
it does not contain an arc $a$ such that
$tail(a)$ and $head(a)$ are both hybrid vertices.

\begin{theorem}\label{thtb}
Let $N$ be a stable phylogenetic network on $X$ and let $T$ 
be a phylogenetic tree on $X$. Then $T$ is a base tree for $N$ if 
and only if there exists an $X$-set $C$ of $\U(N)$ that endorses $T$ 
and the map $\overline{\xi_{C}^+}$ is bijective.
\end{theorem}

\begin{proof}
Without loss of generality,
we may assume that $N$ is not a phylogenetic tree as otherwise the result is
trivial. Since $N$ is stable, we may assume without loss of
generality that there exists a sound  MUL-tree $\M=(M,\mu)$
such that $\M=\U(N)$. Moreover the map $\kappa_{N}:V^-(N)\to V(\M)/\sim$
must be bijective.

Assume first that there exists an $X$-set $C$ of $\M$ 
such that $T$ is endorsed by $\M$ via $C$
and that  $\overline{\xi_C^+}$ is bijective.
Then we may assume without loss of generality
that $T$ is in fact $M_C$.  Suppose for contradiction that $M_C$
is not a base tree of $N$. Then there exists a vertex $v$ in $N$
that is neither contained in $V(M_C)$ nor obtained from
$M_C$ by the 3-step process described above. Since $M_C$ is
displayed by $N$ in view of Theorem~\ref{thfm}
and $f\circ\xi_C^+: V(M_C^+)\to V^-(N)$ is bijective as both $\kappa_N$ and 
$\overline{\xi_C^+}$ are bijective, it follows that
$v$ must be a hybrid vertex of $N$.
Since $N$ is stable and therefore compressed by \cite[Theorem 1]{HMSW16},
it follows that the unique child $c\in V(N)$ of $v$ as well as all
parents of $v$ must be tree vertices of $N$.
Let $p\in V(N)$ denote a parent of $v$. 
Then the bijectivity of $f\circ \xi_C^+: V(M_C^+)\to V^-(N)$
ensures that $p$ and
$c$ are mapped to distinct vertices in $M_C^+$ under $ (f\circ \xi_C^+)^{-1}$.
Note that since the map $ (f\circ\xi_C^+)^{-1}$ is ancestor-preserving,
$c''=(f\circ \xi_C^+)^{-1}(c)$ must be strictly below
$p''=(f\circ \xi_C^+)^{-1}(p)$. Hence,  $c''$ and $p''$
must either be vertices of $M_C$ or are subdivision vertices added to
$M_C$ in the aforementioned 3-step process. Since $M_C$
is displayed by $N$ in view of Theorem~\ref{thfm}, it follows
that $p''$ to $c''$  must
also be vertices of $N$, and that the directed path from $p''$ to $c''$
in $M_C^+$ induces a directed path from  $p''$ to $c''$
in $N$ that must cross $v$. Hence, $T$ must be a base tree for
$N$ which is impossible.

Conversely, assume that $T$ is a support
tree for $N$. Then $T$ is displayed by $N$.
By Theorem~\ref{thfm}, it follows that there
exists an $X$-set $C$ of $\M$ such that $T$ is  endorsed
by $\M$ via $C$ and $\overline{\xi_{C}^+}$ is injective.
It therefore remains to show that $\overline{\xi_{C}^+}$
is surjective. Since we may assume without loss of generality
that $T$ is in fact $M_C$ as $T$ is endorsed by $\M$
via $C$, it suffices to show that the map
$end\circ\Psi_{N}^{-1}\circ \xi_C^+:V(M_C^+)\to V^-(N)$ 
is surjective. 
Suppose that $v\in V^-(N)$. Then $v$ is a tree vertex of $N$.
Let $x\in X$  and let $P_v$ be the directed subpath of $P_{M_C}(x)$
starting at $\rho_N$ and ending in $v$. Since $\Psi_{N}$
preserves the ancestor relationships, it follows that, in 
$\M$,  the vertex $\Psi_{N}(P_v)$
lies on the directed path  from $\rho_{M}$ to the vertex
$\Psi_{N}(P_{M_C}(x))$. Since $M_C$ is a base tree for $N$,
there exists some $u\in V(M_C^+)$ such that
$\xi_C^+(u)=\Psi_{N}(P_v)$. Hence, $\Psi_{N}^{-1}\circ\xi_C^+(u)=P_v$. Since
$v$ is the end vertex of $P_v$ it follows that
$end\circ\Psi_{N}^{-1}\circ\xi_C^+(u)=v$. Thus, 
$end\circ\Psi_{N}^{-1}\circ\xi_C^+$ is surjective, as required. 
\end{proof}

As in the case of Theorem~\ref{thfm}, the equivalence 
postulated by Theorem~\ref{thtb} does not
hold in case the phylogenetic network $N$ 
is not stable. Consider again the phylogenetic networks
$N$ and $N'=F(\U(N))$ depicted in Figure~\ref{figfun}(i)
and (iii) respectively. 
Then the triplet $23|1$
 is a base tree for $N$ but not for $N'$. Hence,
by Theorem~\ref{thtb}, there exists no
$X$-set $C$ of $\U(N')$ that endorses that triplet
and $\overline{\xi_{C}^+}$ is
bijective.

\section{Stable phylogenetic
  networks that are tree-child or reticulation-visible}
\label{sec:tree-child}

In this section, we turn our attention to
clarifying the relationship between various
key concepts originally introduced for binary
phylogenetic networks. Our particular focus
lies on the notion of a tree-child phylogenetic network \cite{CRV07}
which has attracted a considerable amount of attention
in the literature and also reticulation-visible networks
\cite{HRS10}. These concepts are closely related in that a phylogenetic
network is reticulation-visible whenever it is tree-child
(see e.\,g.\,\cite{St16} in case the network in question is binary).
As an immediate consequence of the main result in this
section (Theorem~\ref{thtc}), we see that for stable
phylogenetic networks both properties are also captured
by our map  $\overline{\xi_C^+}$.


We start with presenting formal definitions for both concepts.
Let $N$ be a phylogenetic network on $X$. We say that an
interior vertex $v$ of
$N$ is \emph{tree-child} if $v$ has a child $u$ that is a tree vertex.
More generally, we say that $N$ is  \emph{tree-child}
if all interior vertices of $N$ are tree-child. As an immediate
consequence of \cite[Corollary 1]{HMSW16} it follows that
any binary phylogenetic network that is tree-child must also be stable.
It should however be noted that the converse is not true in general
(see e.\,g.\,the network
depicted in Figure~\ref{favexpl}(i) which is stable but not
tree-child).

We show next that any tree-child network can also be seen as
a stable phylogenetic network with an extra property in terms of
displaying a phylogenetic tree. To do this, we require a
further concept which we introduce next.
Suppose  $N$ is a phylogenetic network on
$X$ and $w\in \V(N)$. Then we
say that $w$ is a \emph{vertex-stable ancestor}
of a leaf $x \in X$ if $w$ belongs to all directed paths from the
root $\rho_N$ of $N$ to $x$. Note that vertex-stable ancestors were
called ``stable'' ancestors in \cite{HRS10} and that a hybrid vertex
that is a vertex-stable ancestors of some leaf is called ``visible''
in  \cite{HRS10}. Also note that the root of
a phylogenetic network $N$ is a vertex-stable ancestor for each leaf of $N$.

The proof of the following result was outlined in \cite[page 249]{St16}.
We include it for completeness sake as our notion of a phylogenetic network
is slightly different from the one in \cite{St16}.

\begin{lemma}\label{lmtc}
  Let $N$ be a phylogenetic network on $X$. Then $N$ is tree-child if
  and only if for all vertices $v \in \V(N)$, there exists a leaf
  $x_v\in X$ such that $v$ is a vertex-stable ancestor for $x_v$.
  Equivalently, $N$ is tree-child if and only if $N$ is compressed and,
  for all 
  tree-vertices $v \in \V(N)$, there exists a leaf $x_v\in X$
  such that $v$ is a
  vertex-stable ancestor for $x_v$.
\end{lemma}

\begin{proof}
  Assume first that $N$ is tree-child. Let $v\in \V(N)$.
 Then there exists a directed path $P:
  v=v_0, v_1, \ldots, v_k$ of vertices of $N$ such that $v_k$ is a leaf
 and $v_i$
  is a tree vertex, for all $0 < i \leq k $.
  Hence, no vertex on $P$ other than possibly $v$
  can be a hybrid vertex of $N$. Thus,
  $v$ must be a vertex-stable ancestor of $v_k$.

  Conversely, assume that $N$ is not tree-child.
  We need to show that there exists an interior
  vertex of $N$ that is not a vertex-stable ancestor of any leaf of $N$.
  To see this note that since $N$ is not tree-child
  there must exist some vertex  $v\in \V(N)$ that is not tree-child.
Hence,  all children
  of $v$ must be hybrid vertices of $N$. Let $x$ be a leaf of $N$ that is 
  below $v$. Then since a hybrid vertex of $N$ cannot be a leaf
  of $N$ there exists a child $w\in \V(N)$ of $v$ that is an ancestor of $x$.
  Since $w$ must be a hybrid vertex by the choice of $v$,
  there must exist a directed path from
  the root $\rho_N$ of $N$ to $x$ that does not cross $v$.
  Since this argument applies to all leaves below $v$, it follows that
  there exists no leaf of $N$ for which $v$ is a vertex-stable ancestor,
  as required.

  To see the second part of the lemma, it suffices to note
  that  the (unique) child of a hybrid vertex in a compressed phylogenetic
  network must be a tree vertex. The stated characterization then follows from
  the straight-forward observation that a hybrid vertex $h$ is a
  vertex-stable ancestor of some leaf $x$ of $N$ if and only if the child of
  $h$ is a vertex-stable ancestor of $x$.
\end{proof}

In a stable phylogenetic
network, tree vertices that are vertex-stable ancestors for
a given leaf can be characterized using the following proposition. To
state it, we denote for an $X$-set $C$
of a MUL-tree $\M$ the last common ancestor of
the elements of $C$ in $\M$ by $r_C$. Note that the pre-image of $r_C$
under $\xi^+_C$ is the root of
$M_C^+$. By abuse of terminology, we denote that root also by
$r_C$.
Also, for two non-empty sets $A$ and $B$
and a map $g:A\to B$, we denote by $g(A)$ the image of $A$ under $g$.

\begin{proposition}\label{prstan}
  Let $N$ be a stable phylogenetic network on $X$ and
  let $v$ be a
  tree vertex in $\V(N)$.
  Then, there exists a leaf $x_v$ of $N$ such that $v$ is a
  vertex-stable ancestor
  of $x_v$ if and only if, for all $X$-sets $C$ of $\M=\U(N)$, 
one of the following two properties
  holds:
\begin{itemize}
\item[(i)] $\kappa_{N}(v)$ belongs to the image set of $V(M_C^+)$
  under $\overline{\xi_C^+}$.
\item[(ii)] $v$ is an ancestor of $\kappa_{N}^{-1}(\overline{r_C})$ in $N$. 
\end{itemize}  
\end{proposition}

Rather than continuing with the proof, we next
illustrate Properties~(i) and (ii) for the convenience of the reader.
Consider the phylogenetic
network $N'$ and the MUL-tree $\mathcal M=\mathcal U(N')$ depicted in
Figures~\ref{figfun}(iii) and (ii) respectively. Let $\rho_{N'}$ and
$\rho_{\mathcal M}$
be the root of $N'$ and $\mathcal M$ respectively. Clearly, $\rho_{N'}$
is a vertex-stable ancestor of some leaf of $\mathcal M$.
Consider now the four $X$-sets $C$ of $\mathcal M$. Three of these
$X$-sets $C$ are such that $\rho_M=r_C$ and that, therefore,
$\rho_M$ belongs to the vertex set of $M_C$. In this case
$\overline{\xi_C^+}(\rho_M)$ coincides with $\kappa_{N'}(\rho_{N'})$.
For the fourth $X$-set $C$ i.\,e.\,the one comprising
the three right-most leaves of $\mathcal M$, we have that $r_C=u$.
Since $\kappa_{N'}^{-1}(\overline u)$ is a child of $\rho_{N'}$ in $N'$,
it follows that $\rho_{N'}$ satisfies Property~(ii).

\begin{proof}
Assume first that there exists a leaf $x_v$ of $N$ such
that $v$ is a vertex-stable ancestor of $x_v$. Let $C$ be
an $X$-set of $\M=(M,\mu)$. 
  It suffices to show that if $v$ is not an ancestor of
  $\kappa_{N}^{-1}(\overline{r_C})$ in $N$ then
  $\kappa_{N}(v)\in \overline{\xi_C^+}(V(M^+_C))$.

  Assume for contradiction that this is not the
  case, that is, $v$ is not an ancestor of
  $\kappa_{N}^{-1}(\overline{r_C})$
  in $N$ and  $\kappa_{N}(v)\not \in \overline{\xi_C^+}(V(M_C^+))$.
  Then there exists no vertex $v'\in V(M_C^+)$ such that
  $v$ is contained in the equivalence class generated
  by $v'$ under the projection $p_{\M}: V(\M)\to V(\M)/\sim$. Hence,
  there exists no element
  $x \in X$ and leaf $l\in C$ with $l\in\mu(x)$
  such that the directed subpath of
  $P_{M_C^+}(l)$ of $N$
  that starts at $f(r_C)$ and ends in $x$ contains $v$.
Combined with the fact that, also by assumption,  $v$  is not
an ancestor of $\kappa_{N}^{-1}(\overline{r_C})$ in $N$ it
follows that there exists no element $x\in X$ 
and leaf $l\in C$ with $l\in\mu(x)$ such
that $v$ is a vertex on the directed path $P_{M_C^+}(l)$.
Consequently, for all $x \in X$, there exists a directed path
in $\pi_x(N)$ that does not cross $v$. But then
$v$ cannot be a vertex-stable ancestor of some element
in $X$ which is impossible.

Conversely, assume that $v$ is not a vertex-stable ancestor
of some element of $X$, that is, 
for every element $x \in X$, there exists a path $P_x\in \pi_x(N)$
that does not cross $v$. We need to show that there exists some
$X$-set $C$ of $\M$ such that
neither Property~(i) nor Property~(ii) of the proposition holds.

Put $C:=\bigcup_{x\in X}\{P_x\}$. Clearly $C$ is an $X$-set of
$\M$. For a vertex $w\in V(M_C^+)$ denote by $x_w$
a leaf of $M_C^+$ below $w$. Since $P_{M_C^+}$ is the map that
assigns to a vertex $w\in V(M_C^+)$
the subpath of $P_{x_w}$ in $N$ that starts at $\rho_N$ and ends
in $w$ it follows that the end-vertex of $P_{M_C^+}(r_C)$
is $\kappa_{N}^{-1}(\overline{r_C})$. Thus, $v$ is not an
ancestor of $\kappa_{N}^{-1}(\overline{r_C})$. So Property~(ii) does not hold.
Moreover, by construction, no path in $P_{M_C^+}(V(M_C^+) )$
can contain $v$ as a vertex.  Hence, there cannot exist some
$v'\in V(M_C^+)$ such that $P_{M_C^+}(v')$ ends in $v$.
But then there cannot exist some
$v'\in V(M_C^+)$ such that
$\kappa_N(v)=\kappa_N\circ end\circ P_{M_C^+}(v')=\overline{\xi_C^+}(v')$. Thus,
$\kappa_{N}(v)\not \in \overline{\xi_C^+}(V(M_C^+))$
Hence, Property~(i) does not hold either. This concludes the proof.
\end{proof}

To help illustrate  the consequences of Proposition~\ref{prstan}, we
need to introduce further terminology. For $\M$ a MUL-tree and
$C$ an $X$-set of $\M$, we denote by $V(\M)^C$ the subset of $V(\M)$
containing $r_C$ and all vertices $v$ of $\M$ such that no vertex
$u\in V(\M)$ below $v$ satisfies $u \sim r_C$. In other words, the set $V(\M)^C$
is precisely the set of vertices $v$ of $\M$ such that
the vertex $f(r_C)$ is not below
the vertex  $f(v)$, plus the vertex $r_C$ itself.
Note that $V(\M)^C=V(\M)$ if and only if $r_C$ is the root of $\M$.

To illustrate this definition, consider the MUL-tree $\M$
depicted in Figure~\ref{figfun}(ii), and the $X$-set $C$ of $\M$
that consists of the three leaves below the vertex labelled
$u$. Then $r_C=u$ and $V(\M)^C$ contains all vertices below $u$ and all
 vertices below $v$ because no vertex $v'\in V(\M)-\{v\}$
below $v$ satisfies $v' \sim r_C$. However, the root $\rho_{\M}$ of $\M$ is
not contained in $V(\M)^C$, since $r_C$ is below $\rho_{\M}$ 
and $r_C\sim r_C$ clearly holds.
Furthermore, for $\M$ the MUL-tree depicted in
Figure~\ref{favexpl}(ii) and $C$ the $X$-set that consists of the
leaves of the dashed tree in that figure, we have that $V(\M)^C=V(\M)$.

Armed with this notation and Proposition~\ref{prstan}, we are now
ready to characterize stable phylogenetic networks that are also tree-child
or reticulation-visible in terms of our map $\overline{\xi_C^+}$.
We begin with a
straight-forward result that will turn out to be useful in this
context.

\begin{lemma}\label{lmim}
  Let $N$ be a stable phylogenetic  network on $X$ and
  let $C$ be an $X$-set of $\M=\U(N)$.
  Then $\overline{\xi_C^+}(V(M_C^+))\subseteq V(\M)^C/\sim$.
\end{lemma}

\begin{proof}
  Assume for contradiction that there exists a vertex $v \in V(M_C^+)$
  such that $\overline{\xi_C^+}(v) \notin V(\M)^C/\sim$. Then,
  the end vertex of the directed path
  $\Psi_{N}^{-1}(\xi_C^+(\rho_C))$ is below the
  end vertex of the directed path $\Psi_{N}^{-1}(\xi_C^+(v))$.
  But this is impossible since
   $end \circ \Psi_{N}^{-1}\circ \xi_C^+$
  must preserve the ancestor relationships in $M_C^+$ as both
  the maps $\xi_C^+$ and $end \circ \Psi_{N}^{-1}$ preserve them.
  \end{proof}

We are now ready to state the main result of this section
which provides characterizations for when a stable phylogenetic
network is tree-child or reticulation-visible.  A
phylogenetic network $N$ is called \emph{reticulation-visible} if every
hybrid vertex $h$ of $N$ is a vertex-stable ancestor of
some leaf $x_h$ of $N$ \cite{HRS10} -- see
e.\,g.\, \cite{BS16, GGLVZ18} for more
on such networks. Note that Lemma~\ref{lmtc}
 implies that a tree-child network is also reticulation-visible. However
 note that, in general,
 a reticulation-visible network need neither be stable nor tree-child (see
 for example the phylogenetic network depicted in Figure~\ref{figfun}(i)).

\begin{theorem}\label{thtc}
Let $N$ be a stable phylogenetic network on $X$. We have:
\begin{itemize}
\item[(i)] $N$ is tree-child if and only if for all $X$-sets $C$ of
  $\M=\U(N)$
  the sets $\overline{\xi_C^+}(V(M_C^+))$ and $V(\M)^C/\sim$ coincide.
\item[(ii)] $N$ is reticulation-visible if and only if for all $X$-sets $C$ of
$\M=(M,\mu):=\U(N)$ and all non-root vertices $v$ of $\M$  for
which $|p_{\M}(u)| < |p_{\M}(v)|$ holds for the parent $u$ of $v$, we have 
$p_{\M}(v) \in
\overline{\xi_C^+}(V(M_C^+))$.
\end{itemize}
\end{theorem}

\begin{proof}
   (i): Without loss of generality we may assume that $N$
   is not a phylogenetic tree. Then since $N$ is stable, $N$ must be
   compressed by \cite[Theorem 1]{HMSW16}.
   By the second part of Lemma~\ref{lmtc},
   it follows that $N$ is tree-child if and only if for all
   tree-vertices $v\in \V(N)$ there exists a leaf $x_v\in L(N)$ such that
  $v$ is a vertex-stable ancestor of $x_v$. In combination 
  with Proposition~\ref{prstan}, this implies
  that $N$ is tree-child if and only if for all $X$-sets $C$
  of $\M$ and all tree vertices $v\in\V(N)$, we either have
  that $v$ is an ancestor of $\kappa_{N}^{-1}(\overline{r_C})$
  in $N$ or that $\kappa_{N}(v)$ belongs to
  $A_C:=\overline{\xi_C^+}(V(M_C^+))$.

  Assume first that $N$ is tree-child. Let 
  $C\subseteq L(\M)$ denote an $X$-set of $\M$. In view
  of Lemma~\ref{lmim} it suffices to show that
   $V(\M)^C/\sim$ is contained in $A_C$.
  Let $b\in V(\M)^C/\sim$. 
  Then there exists some $a\in V(\M)^C\subseteq V(\M)$ such 
  that $p_{\M}(a)=\overline{a}=b$.
  Since $\Psi_{N}^{-1}(a)\in \pi^-(N)$ it follows
  that $a':=end\circ \Psi_{N}^{-1}(a)$ is a tree vertex of $N$.
  
  If $r_C\not=a$ then the definition of $V(\M)^C$ 
  implies that $\kappa^{-1}_{N}(\overline{r_C})
  =end\circ \Psi_{N}^{-1}(r_C)$ is not below $a'$. Since $end\circ \Psi_{N}^{-1}$
  is ancestor preserving, it follows that  $a'$ is not an
  ancestor of $\kappa^{-1}_{N}(\overline{r_C})$.
  Combined with our observations
  at the beginning of this proof, it follows that 
  $\kappa_{N}(a')\in A_C$ in case $a'$ is not a leaf of $N$.
 Hence, there exists some  $d\in V(M_C^+)$ such that
  $b=\overline{a}=\kappa_{N}\circ end \circ
  \Psi_{N}^{-1}(a)=\kappa_{N}(a')=\overline{\xi_C^+}(d)$ in this case.
  If $a'$ is a leaf of $N$ then $a'$ is a leaf of $M_C^+$ and, so, 
  $\overline{\xi_C^+}(a')=\overline{a}=b$.

  If $r_C=a$ then $b=\overline{a}=\overline{r_C}=\overline{\xi_C^+}(r_C)$.
  In summary, we obtain that $V(\M)^C/\sim$ is
  contained in $A_C$ whenever $N$ is
  tree-child.

To see the converse implication it suffices to 
show that if $N$ is not tree-child
then there must exist some $X$-set $C\subseteq L(\M)$ such that 
$V(\M)^C/\sim$ is not contained in $A_C$. So assume that $N$ is
not tree-child. Then there must exists some tree vertex
$w\in \V(N)$ such that both its children $w_1$ and $w_2$ are
hybrid vertices of $N$. Without loss of generality,
we may assume that $w$ is as far away from the root $\rho_{N}$ of
$N$ as possible. Note  that  $w_1\not=w_2$ as
$N$ is a phylogenetic network and 
that  $w$ is the only parent shared by $w_1$ and $w_2$ as $N$ is stable.
For $i=1,2$, let $P_i\in\pi_{w_i}(N)$ 
denote a directed path form $\rho_{N}$ ending in $w_i$
that does not contain the arc $(w,w_i)$ (which must exist even if the
other parent of $w_i$ is below $w$). Let $C\subseteq L(\M)$ denote an 
$X$-set such that $M_C^+$  contains $P_1$ and $P_2$ in its 
vertex set. Let $P_w\in\pi_w(N)$ denote a directed path
from $\rho_{N}$ to $w$. Then $P_w\in V(\M)$ and the choice of $P_1$ and
$P_2$ implies that $f(r_C)=end\circ \Psi^{-1}_{N}(r_C)$ is 
not below 
$end\circ \Psi^{-1}_{N}(P_w)= f(P_w)$ in $N$. By definition, it follows
that $P_w$ is a vertex in $V(\M)^C$. Hence,
$\overline{P_w}$ is a vertex in $V(\M)^C/\sim$. However, by
construction, there cannot exist some $w'\in V(M_C^+)$
such that $\xi_C^+(w')=P_w$ and, therefore, that
$\overline {\xi_C^+}(w')=\overline{P_w}$. Thus, $V(\M)^C/\sim$
is not contained in $A_C$, as required.

(ii): Note first that since $N$ is stable, and thus 
compressed by \cite[Theorem 1]{HMSW16}, the (unique) child of
a hybrid vertex of 
$N$ is a tree vertex. Hence, $N$ is reticulation-visible 
if and only if for all vertices $w$ of $N$ whose parent 
is a hybrid vertex, there exists a leaf $x_w$ of $N$ such 
that $w$ is a vertex-stable ancestor of $x_w$. Assume
that $w$ is a tree vertex of $N$ whose parent $v$ is a hybrid
vertex.  

Assume first that $w$ is a vertex-stable ancestor of some
leaf $x_w$ of $N$. Let $C$ denote an $X$-set of $\M$.
Then, $(\xi_C^+\circ f)^{-1}(w)$
must be a vertex of $M_C^+$. Hence,
$\overline{\xi_C^+}\circ(\xi_C^+\circ f)^{-1}(w)=\kappa_N(w)$.
Consequently, $p_{\M}(f^{-1}(w))=\kappa_N(w)\in \overline{\xi_C^+}(V(M_C^+))$.
Let $w'\in V(M)$ denote the parent of $f^{-1}(w)$ in $\M$ such
that the directed path $P$ from $f(w')$ to $w$ in $N$ crosses $v$.
To see that $|p_{\M}(w')|< |p_{\M}(f^{-1}(w))|$, note first that
$|p_{\M}(w')|\leq |p_{\M}(f^{-1}(w))|$  must always hold. If 
$|p_{\M}(w')|= |p_{\M}(f^{-1}(w))|$ held then $v$ cannot be a vertex
on $P$ which is impossible.

Next, assume for contradiction that $w$ is not a vertex-stable
 ancestor of any leaf of $N$ but that for any $X$-set $C$ of $\M$
 and all non-root vertices $v$ of $\M$  for
which $|p_{\M}(u)| < |p_{\M}(v)|$ holds for the parent $u$ of $v$
we have  $p_{\M}(v) \in \overline{\xi_C^+}(V(M_C^+))$. Let 
$w_1$ and $w_2$ denote the
 two children of $w$. Then using the same notation as in the proof
of the converse implication in Part~(i), it follows that
the image $\overline{P_w}$ of the directed path $P_w=f^{-1}(w)$
 under the projection
$p_{\M}$ must be a vertex in $V(\M)^C/\sim$. 
Let $w'$ be a vertex in $\M$ that is a parent of $w$ below $r_C$.
Then $|p_{\M}(f^{-1}(w))|\not=|p_{\M}(w')| $ must hold since
the directed path from from $f(w')$ to $w$ must contain $v$.
Since  $|p_{\M}(w')|\leq |p_{\M}(f^{-1}(w))| $ always holds,
our assumption implies  that
$\kappa_N(w)=\overline{P_w}=p_{\M}(f^{-1}(w))$ is contained in
$\overline{\xi_C^+}(V(M_C^+))$. In view of Proposition~\ref{prstan} 
it follows that $w$ is a vertex stable ancestor of some leaf of $N$
which is impossible.


\end{proof}

As an immediate consequence of
Theorem~\ref{thtc} we obtain the following result
where we say that a phylogenetic tree
$T$ on $X$ is {\em strongly displayed}
by a phylogenetic network $N$ on $X$ if $T$ is displayed by $N$
and the root of $T$ is the root of $N$. We remark in
passing that for binary
phylogenetic networks this result was also shown as part of
\cite[Theorem 1.1]{Se16}. 

\begin{corollary}\label{cor:strongly-display}
  Let $N$ be a stable phylogenetic
  network on $X$ that is tree-child. Then, every phylogenetic
  tree on $X$ that is strongly displayed by $N$ is a 
  base tree for $N$.
\end{corollary}

\begin{proof}
  Let $T$ be a phylogenetic tree that 
  is strongly displayed by 
  $N$. Then $T$ is clearly displayed by $N$. Since $N$ is 
  stable, Theorem~\ref{thfm} implies  that there
  exists an $X$-set $C$ of $\M=(M,\mu):=\U(N)$
  such that $\overline{\xi_C^+}$ is injective,
  $\xi_C^+(\rho_{T})$ is the root of $\M$, and $T$ and $M_C$
  are isomorphic. 
  Hence, the sets $V(\M)^C/\sim$ and $V(\M)/\sim$ coincide.
  Since $N$ is tree-child,
  Theorem~\ref{thtc} implies that the sets
  $\overline{\xi_C^+}(V(M_C^+))$ and $V(\M)^C/\sim=V(\M)/\sim$ 
  are also equal.
  Since $\overline{\xi_C^+}$ is injective with image set 
  $V(\M)/\sim$, it follows by Theorem~\ref{thtb} that $T$ is a 
  base tree for $N$.
\end{proof}

Note that the converse of Corollary~\ref{cor:strongly-display}
is not true in general, even if the phylogenetic
network $N$ considered is stable. For 
example, all phylogenetic trees strongly displayed by the 
stable phylogenetic network $N$ depicted in 
Figure~\ref{favexpl}(i) are base trees for $N$, 
but $N$ is not tree-child. 
The reason for this is that if $N$ is a stable phylogenetic network 
on $X$ and $C$ is an $X$-set of $\M=(M,\mu):=\U(N)$ such
that $\overline{\xi_C^+}$ 
is injective but not bijective, then this does not imply that 
(although being  isomorphic with $M_C$) $T$  
is not a base tree for $N$. Indeed, 
there may exist a further $X$-set $C'$ of $\M$ such that $T$
and $M_{C'}$ are isomorphic and 
$\overline{\xi_{C'}^+}$ is bijective. By Theorem~\ref{thtb}, this implies 
that $M_{C'}$ (and therefore also $T$) is a base tree for $N$.

\section{Some remarks about reconstructing stable phylogenetic networks}

\label{sec:triplets-trinets}
In this section, we briefly turn our attention to the problem of
constructing stable $X$-networks from induced substructures.
For this, we focus on so called trinets and also certain subtrees of
the MUL-tree obtained as the un-fold of a phylogenetic network.
The reason for this is
that in \cite{HM13} it was pointed out that, in general, a binary phylogenetic
network $N$ (and therefore also an $X$-network) 
is not {\em encoded} that is, up to equivalence, 
uniquely reconstructible
from its induced set of triplets or set of displayed phylogenetic trees.
Given the strong relationship between a phylogenetic
network $N$ and its induced MUL-tree we start with subtrees
induced by $\U(N)$. 

Following \cite{HES14}, we call a MUL-tree on three leaves a 
{\em MUL-triplet}. So for example for $Y=\{1,2\}$, the
MUL-tree $(M,\mu)$ with leaf set $\{a,b,c\}$ and $lca_{M}(a,b)$
strictly below $lca_{M}(a,c)$ and $\mu(1)=\{a,c\}$ and $\mu(2)=\{b\}$
is a MUL-triplet on $Y$. Saying that a MUL-tree
$\M$ {\em displays} a MUL-triplet $\tau$
if there exists a MUL-tree $\M'$ obtainable from $\M$
by deleting vertices and arcs
(suppressing resulting degree two vertices and, if this
has rendered the root of $\M'$ a vertex of degree one, identifying
that root with its unique child) such 
 that $\M'$ and $\tau$
 are isomorphic, it is easy to check that the MUL-triplet $(M,\mu)$
 constructed in the previous example is
 displayed by the left
 MUL-tree depicted in Figure~\ref{enc}(i)  where we represent each
 leaf by its label.  
 
 Although it is
 straight-forward to check that this definition of displaying
 reduces to the one for phylogenetic trees it should be noted
 that although a phylogenetic tree is encoded by
 its set of  displayed triplets a MUL-tree is in general not
 encoded by its set of displayed MUL-triplets. An example for
 this is furnished by the 
  MUL-trees $\M$ and $\M'$ on $X=\{1,2,3,4\}$ depicted in Figure~\ref{enc} which display the same
  set of MUL-triplets (and this set has 21 elements).
  See also \cite{HLMS08} where similar observations were made
  for the set of {\em splits} (i.\,e.\,bipartitions) 
induced by the edges of an {\em unrooted} MUL-tree
and \cite{GW17} for the set of {\em clusters} (i.\,e.\,the set of leaves 
reachable from a vertex) in a MUL-tree.

\begin{figure}[h]
\begin{center}
\includegraphics[scale=0.6]{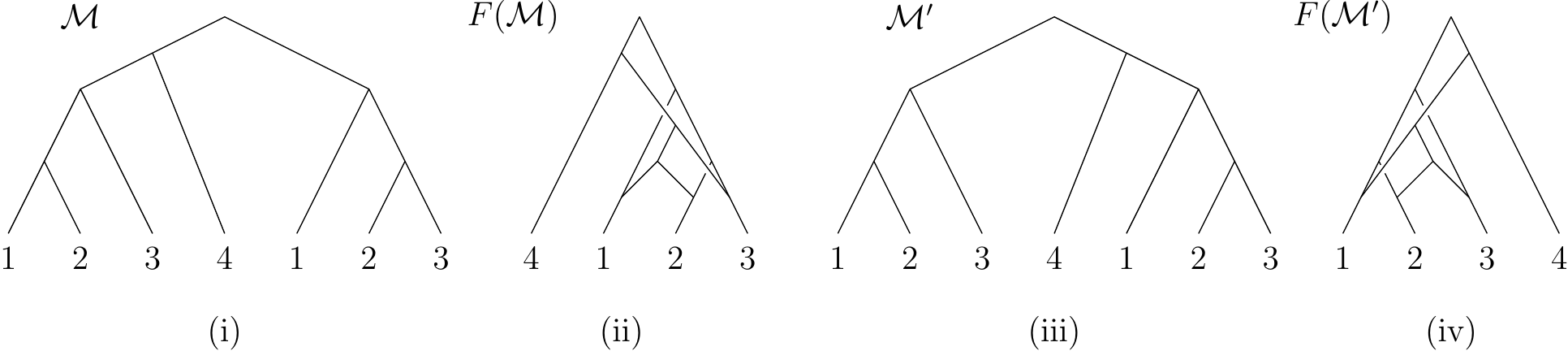}
\caption{Two MUL-trees $\M$ and $\M'$ on $X=\{1,2,3,4\}$
	 displaying the same set of MUL-triplets on $X$. Note 
	 that the MUL-trees $\M_Y$ and $\M'_Y$ are 
	 isomorphic, for all subsets $Y \subseteq X$ of size 3.
For ease of readability, we represent each leaf by its label.}
\label{enc}
\end{center}
\end{figure}


Turning our attention towards the 
two non-isomorphic phylogenetic networks
  $F(\mathcal M)$ and $F(\mathcal M')$
given in Figure~\ref{enc}(ii) and (iv), respectively, it is
easy to check that they induce the same
set of triplets on $X$.
%
%
As a direct consequence of these observations it follows  that,
in general, stable phylogenetic networks
are neither encoded by their set of  induced triplets nor by
the set of MUL-triplets induced by their un-folds.

We next turn our attention to understanding the
encoding potential of {\em trinets},
that is, $X$-networks on three leaves \cite{HM13}.
Collections of trinets displayed by a tree-child
network or so called {\em level-2 networks} are known to 
encode that network \cite{vIM13}. Since, in general, trinet systems
do not encode phylogenetic networks \cite{HIMW15}
it is interesting
to understand whether stable $X$-networks are encoded by the
trinet system they display. To help shed light into
this question, we require a certain 
\emph{leaf-removing} operation for
$X$-networks, which we now describe.

Let $N$ be an $X$-network, and let $l\in X$. Then
the leaf-removing operation applied to $l$ consists of
 repeatedly applying the following three steps until 
 a valid $(X\backslash \{l\}$)-network is obtained:
\begin{itemize}
\item[(a)] Remove the arc incident with $l$.
\item[(b)] Remove resulting hybrid vertices of outdegree
  zero (if any), and their incident arcs.
\item[(c)] Suppress resulting tree vertices of indegree and
 outdegree one (if any).
\end{itemize}

Now, let $Y$ be a subset of $X$ of size three or more, 
and let $N$ be an $Y$-network. Then 
we define the \emph{subnetwork $N_Y$ of 
	$N$ induced by $Y$} to be 
the $Y$-network obtained from $N$ by applying the 
leaf-removing operations to all leaves of $N$ 
in $X\backslash Y$. If $Y$ is of size three, we say that a 
trinet $\tau$ on $Y$ is \emph{displayed} by $N$ 
if $\tau$ and $N_Y$ 
 are isomorphic. Note again that 
 if $N$ is in fact a 
 phylogenetic tree, then this definition of displaying 
 reduces to the one for triplets.

By requiring in step (a) that all leaves of a MUL-tree
$(M,\mu)$ contained in $\mu(l)$, $l\in Y$, are removed, the
aforementioned  leaf-removing operation can also be extended to
a leaf-removing operation for MUL-trees
  $\M$ to obtain a new MUL-tree $\M_Y$. 
 In this case, we refer to $\M_Y$ as a \emph{sub-$k$MUL-tree} of $\M$, 
 where $k=|Y|$. Summarized in
  Observation~\ref{usupp}, we have that the un-fold 
  and leaf-removing operations are commutative:

\begin{observation}\label{usupp}
Let $N$ be a phylogenetic network on $X$, and $Y$ a subset 
of $X$ of size three or more. Then the MUL-trees $\U(N_Y)$ 
and $\U(N)_Y$ are isomorphic.
\end{observation}

Perhaps unsurprisingly, the equivalence in
 Observation~\ref{usupp}
does not hold if the un-fold operation is replaced by the
fold-up operation. The MUL-tree $\M$ 
on $X=\{1,2,3,4\}$ depicted in
Figure~\ref{figsf}(i) furnishes an example for this. 
For $Y$=\{1,2,4\}, we have that
  $F(\M)_Y$ is the phylogenetic network 
  depicted in Figure~\ref{figsf}(ii), whereas $F(\M_Y)$ 
  is the $Y$-network depicted in
   Figure~\ref{figsf}(iii).

\begin{figure}[h]
\begin{center}
\includegraphics[scale=0.6]{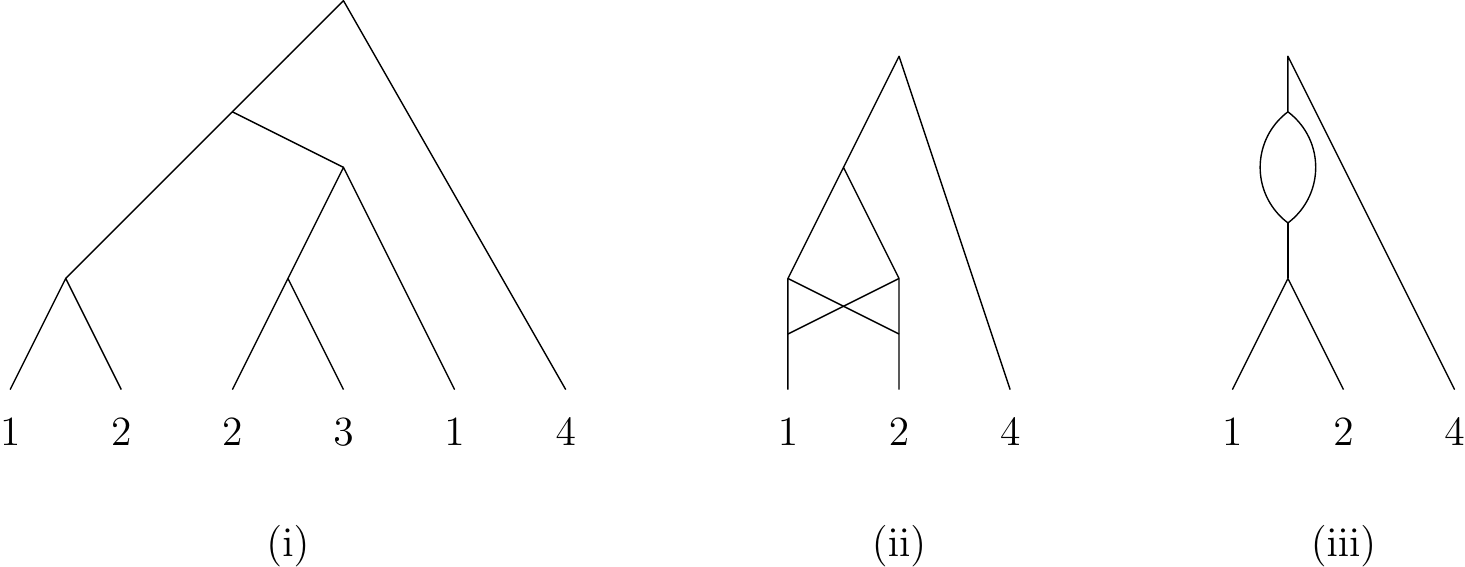}
\caption{(i) A MUL-tree $\M$ on $X=\{1,2,3,4\}$. 
	(ii) The phylogenetic network on $Y=\{1,2,4\}$ obtained 
	by first applying the fold-up operation to $\M$ and 
	then the leaf-removing 
	operation to the leaf labelled $3$. (iii) The
	 $Y$-network obtained by 
	 first applying the leaf-removing operation to the 
	 leaf labelled $3$ and then applying the fold-up operation
	 to the resulting MUL-tree.}
\label{figsf}
\end{center}
\end{figure}

Again, the MUL-trees $\M$ and $\M'$ depicted in 
Figure~\ref{enc} indicate that a MUL-tree is not, in 
general, encoded by the set of its displayed 
sub-3MUL-trees. Indeed, for all $Y \subset \{1,2,3,4\}$ 
of size three, the MUL-trees $\M_Y$ and $\M'_Y$ are 
isomorphic. Note also that the example of these two MUL-trees
implies that their respective fold-ups into the
phylogenetic networks $F(\M)$ and $F(\M')$
are not encoded by their sets of induced trinets.


\section{Discussion and open problems}
\label{sec:discussion}

Stable phylogenetic networks constitute an interesting
class of phylogenetic networks that have already
proven useful for better understanding how
polyploidy in, for example plants, has arisen. By investigating the popular
tree-based, tree-child, and reticulation-visible
properties for such networks we provide novel characterizations
that shed light into the structural
complexity of such networks. Although we only state
our main results (i.e.\,Theorems~\ref{thfm}, \ref{thtb}
and \ref{thtc}) for stable phylogenetic networks, we remark in passing that
they also apply to the more general $X$-networks
(as long as the root of the network is not the tail of two parallel arcs
in the case of Theorem~ \ref{thtb}).
Despite this, many open questions remain that might be of interest for future
study. These include the following:

\begin{enumerate}
\item In their current form, stable phylogenetic networks are defined 
via a fold-up operation for
MUL-trees. From an application point of view this
is somewhat unsatisfactory  as it
is not always clear how to construct such a tree 
\cite{HLMS08}. Thus it might
be of interest to find combinatorial characterizations of stable
phylogenetic networks
that do not rely on that operation. Given our observations
in Section~\ref{sec:triplets-trinets}, this will require
 alternative techniques as
neither triplets, trinets, or general phylogenetic trees
can be used for this. 

\item How can we quantify how close two stable phylogenetic
  networks are. Given that we obtained in Theorem~\ref{thtb}
a characterization for when a phylogenetic tree that is endorsed by a
MUL-tree $\M$ is in fact a base tree of $F(\M)$
and that tree-based networks are defined via
base trees, it might be interesting to see if the similarity
measures defined in \cite{FSS18} can be used for this.
\end{enumerate}

\section*{Acknowledgement}
The authors would like to thank the editor and the
anonymous referee for their constructive comments and suggestions
to improve the paper. GS would like to thank the University of East
Anglia for hosting him during part of this work.

\bibliographystyle{elsarticle-harv}

\bibliography{bibliography}{}

\end{document}